\documentclass[11pt]{article}
\usepackage{amsmath, amssymb, amsthm}
\usepackage{graphicx}
\usepackage{epsfig}
\usepackage{epstopdf}
\usepackage{mathrsfs}

 \newtheorem{thm}{Theorem}
 
 \newtheorem{lem}[thm]{Lemma}
 \newtheorem{prop}[thm]{Proposition}
 \newtheorem{obs}[thm]{Observation}
 \theoremstyle{definition}
 \newtheorem{defn}[thm]{Definition}
 \newtheorem{exmp}[thm]{Example}

\numberwithin{equation}{section}

\newcommand{\Real}{\mathbb{R}}
\newcommand{\set}[1]{\left\{#1\right\}}
\newcommand{\Set}[2]{\set{#1\ \vert\ #2}}
\renewcommand{\atop}[2]{\genfrac{}{}{0pt}{}{#1}{#2}}

\begin{document}

\title{A polynomial representation and a unique code of a simple undirected graph}
\author{ Shamik Ghosh\thanks{Department of Mathematics, Jadavpur University, Kolkata-700032, India. sghosh@math.jdvu.ac.in (Communicating author)}\,,
Raibatak Sen Gupta\thanks{Department of Mathematics, Jadavpur University,
 Kolkata-700032, India. raibatak\_sengupta@yahoo.co.in}\,,
M. K. Sen\thanks{Department of Pure Mathematics, University of Calcutta, Kolkata-700019, India. senmk6@yahoo.com}}

\date{}
\maketitle

\begin{abstract}
\noindent
In this note we introduce a representation of simple undirected graphs in terms of polynomials and obtain a unique code for a simple undirected graph.

\vspace{0.5em}\noindent
{\small {\bf AMS Subject Classifications:} 05C62, 05C60.

\noindent
{\bf Keywords:} Graph, simple graph, undirected graph, graph representation, graph isomorphism.}
\end{abstract}

\noindent
Let $M$ be the set of all positive integers greater than $1$. Let $n\in M$ and $V(n)$ be the set of all divisors of $n$, greater than $1$. Define a simple undirected graph $G(n)=(V,E)$ with the vertex set $V=V(n)$ and any two distinct vertices $a,b\in V$ are adjacent if and only if $\gcd (a,b)>1$. From an observation in \cite{CGMS} it follows that any simple undirected graph is isomorphic to an induced subgraph of $G(n)$ for some $n\in M$. For the sake of completeness and further use of the construction we provide a sketch of the proof below. Throughout the note by a {\em graph} we mean a simple undirected graph.

\begin{thm}\label{t1}
Let $G$ be a graph. Then $G$ is isomorphic to an induced subgraph of $G(n)$ for some $n\in M$.
\end{thm}

\begin{proof}
Let $G=(V,E)$ be a graph. Let $\set{C_1,C_2,\ldots,C_k}$ be the set of all maximal cliques of $G$. For $i=1,2,\ldots,k$, let $p_i$ be the $i^\text{th}$ prime. For each $v\in V$, define $s_1(v)=\prod\Set{p_j}{v\in C_j}$. Now in order to make the values of $s_1(v)$ distinct for distinct vertices we modify $s_1(v)$ by using different powers of primes $p_j$, if required. For each $v\in V$, let $s(v)$ be the modified value of $s_1(v)$. Let $n$ be the least common multiple of $\Set{s(v)}{v\in V}$. Now it is clear that for any $u,v\in V$,
$$\begin{array}{cl}
 & u \text{ is adjacent to  } v \text{ in } G \Longleftrightarrow u,v\in C_i \text{ for some } i\in\set{1,2,\ldots,k}\\
\Longleftrightarrow & p_i \text{ is a factor of both } s(u) \text{ and } s(v) \text{ for some } i\in\set{1,2,\ldots,k}\\
\Longleftrightarrow & \gcd (s(u),s(v))>1 \Longleftrightarrow u \text{ is adjacent to  } v \text{ in } G(n).\\

\end{array}$$
Thus $G$ is isomorphic to the subgraph of $G(n)$ induced by the set $\Set{s(v)}{v\in V}$ of vertices of $G(n)$.
\end{proof}

\begin{exmp}\label{e1}
Consider the graph $G=(V,E)$ in Figure \ref{fig1}. The maximal cliques of $G$ are $C_1=\set{v_1,v_2,v_3,v_4}$, $C_2=\set{v_2,v_3,v_4,v_5,v_6}$, $C_3=\set{v_7,v_8,v_{10}}$, $C_4=\set{v_9,v_{10}}$ and $C_5=\set{v_{10},v_{11}}$. We assign the $i^\text{th}$ prime to the clique $C_i$ for $i=1,2,\ldots,5$. Then for each $v\in V$ we compute $s_1(v)$ and $s(v)$ as in the proof of Theorem \ref{t1}.
$$\begin{array}{|c|c|c|c|c|c|c|c|c|c|c|c|}
\hline
v & v_1 & v_2 & v_3 & v_4 & v_5 & v_6 & v_7 & v_8 & v_9 & v_{10} & v_{11} \\
\hline
s_1(v) & 2 & 2\cdot 3 & 2\cdot 3 & 2\cdot 3 & 3 & 3 & 5 & 5 & 7 & 5\cdot 7\cdot 11 & 11 \\
\hline
s_1(v) & 2 & 6 & 6 & 6 & 3 & 3 & 5 & 5 & 7 & 385 & 11 \\
\hline
s(v) & 2 & 2\cdot 3 & 2^2\cdot 3 & 2\cdot 3^2 & 3 & 3^2 & 5 & 5^2 & 7 & 385 & 11\\
\hline
s(v) & 2 & 6 & 12 & 18 & 3 & 9 & 5 & 25 & 7 & 385 & 11\\
\hline
\end{array}$$

\vspace{-1em}
\begin{figure}[h]
\begin{center}
\includegraphics[scale=0.35]{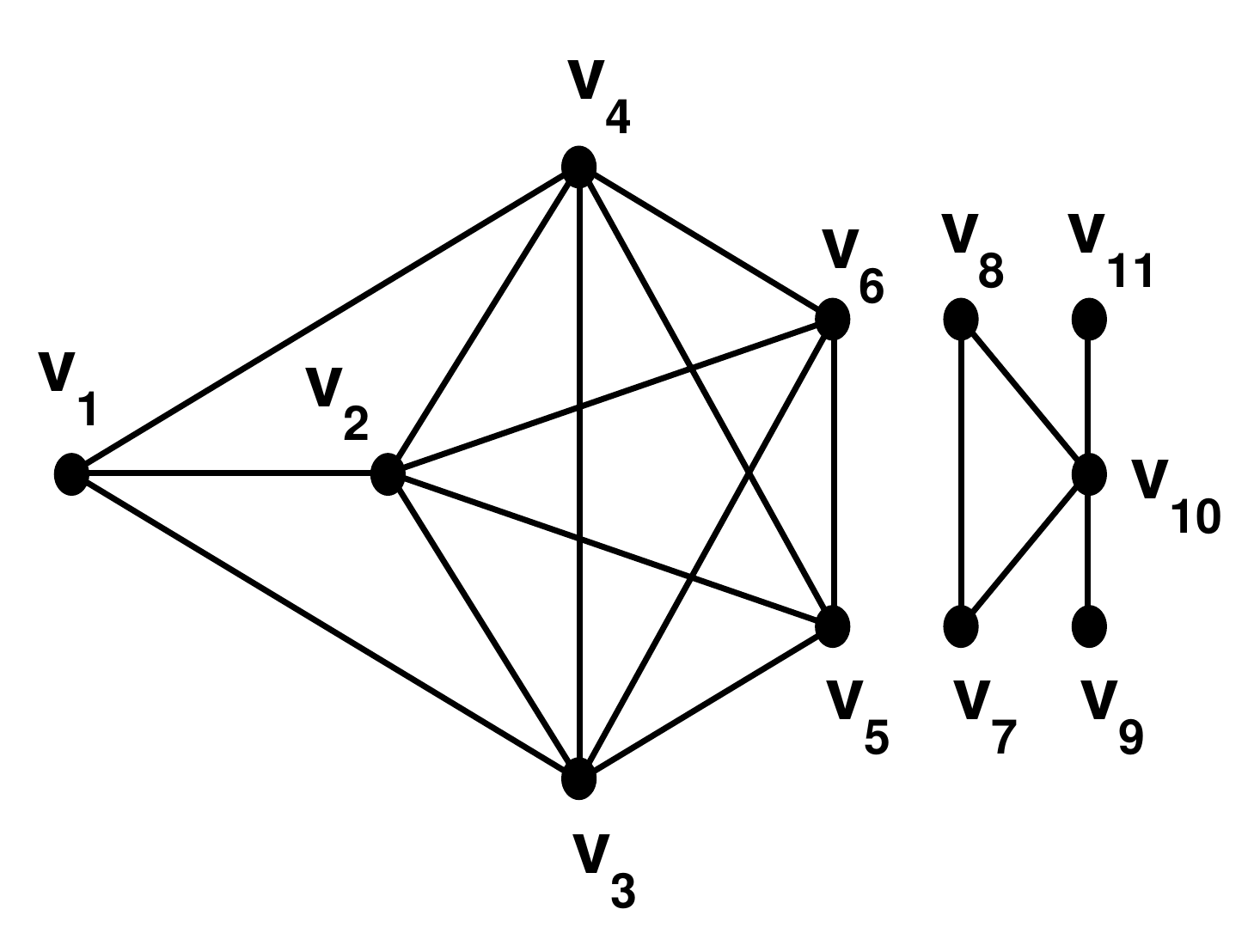}
\caption{The graph $G$ in Example \ref{e1}}\label{fig1}
\end{center}
\end{figure}

\vspace{-1em}
\noindent
So $G$ is isomorphic to the subgraph of $G(n)$ induced by the set $\set{2,3,5,6,7,9,11,12,18,25,385}$ of vertices of $G(n)$, where $n=2^2\cdot 3^2\cdot 5^2\cdot 7\cdot 11=69300$.
\end{exmp}

\noindent
It is important to note that instead of taking all maximal cliques of $G$ in Theorem \ref{t1}, it is sufficient to consider a set of cliques of $G$ which covers both vertices and edges of $G$. 

\begin{defn}\label{deftcc}
Let $G=(V, E)$ be a graph. A set $S$ of cliques of $G$ is called a {\it total clique covering} of $G$ if $S$ covers both $V$ and $E$. The minimum size of a total clique covering of $G$ is called the {\it total clique covering number} of $G$. We denote it by $\theta_t(G)$.
\end{defn}

\noindent
On the other hand, given a finite  sequence of  positive integers, one can construct a simple undirected graph as follows :

\begin{defn}\label{dseq}
Let $\sigma=(a_1, a_2, \ldots, a_m)$ be a finite sequence of positive integers. Then corresponding to this sequence, define a  graph $G[\sigma]=(V, E)$, where $V=\{v_1, v_2, \ldots, v_m\}$, $v_i$ corresponds to $a_i$ ($a_i$ is called the {\em label} of $v_i$) for $i=1, 2, \ldots, m$, and $v_iv_j \in E$ if and only if $i \not=j$ and $\gcd(a_i, a_j) >1$. The graph $G[\sigma]$ is said to be {\em realized} by the sequence $\sigma$.
\end{defn}

\noindent
Now by Theorem \ref{t1}, every graph can be realized by some sequence of positive integers. Also in Definition \ref{dseq}, it is sufficient to take the entries $a_i$ of $\sigma$ square-free if $a_i>1$. Thus for convenience we specify the sequence in the following manner.

\begin{defn}\label{dcseq}
Let $G=(V,E)$ be a graph. Let $|V|>1$ and the graph $G_1$ be obtained from $G$ by deleting isolated vertices of $G$, if there is any. Let $\sigma_1$ be a finite non-decreasing sequence of square-free positive integers greater than $1$ such that $G_1\cong G[\sigma_1]$. If $G$ has $s$ isolated vertices, then we prefix $s$ number of $1$'s in $\sigma_1$ to obtain the sequence $\sigma$. For $s=0$, $\sigma=\sigma_1$. If $G$ is a graph with a single vertex, then $\sigma=(1)$. Then $G\cong G[\sigma]$ and the sequence $\sigma$ is called a {\em coding sequence} of $G$. The entries of $\sigma$, which are greater than $1$, are called {\em non-trivial}. Let $\lambda(\sigma)$ be the least common multiple of all non-trivial entries of $\sigma$.
\end{defn}

\begin{lem}\label{lem1}
Let $\sigma=(\lambda_1,\lambda_2, \ldots, \lambda_m)$ be a coding sequence of a graph $G=(V, E)$. Let $V=\{v_1, v_2, \ldots, v_m\}$, where $v_i$ corresponds to $\lambda_i$ for all $i=1, 2, \ldots, m$ (we write $\lambda_i= \mu(v_i)$). Let $\{p_1, p_2, \ldots, p_k\}$ be the set of all distinct prime factors of $\lambda(\sigma)$.  Suppose $G$ has $s\geqslant 0$ isolated vertices. Define $S_j=\set{v_j}$ for $j=1,2,\ldots,s$ and for each $j=1,2,\ldots,k$, define $S_{s+j}=\Set{v_i\in V}{p_j\text{ divides }\mu(v_i)=\lambda_i}$. Then $S=\Set{S_j}{j=1,2,\ldots,s+k}$ is a total clique covering of $G$ such that $v \in S_{i_1}\cap S_{i_2}\cap \cdots \cap S_{i_r}$ ($i_t>s$, $t=1,2,\ldots,r$) and $v\notin S_j$ for all $j\in\set{s+1,s+2,\ldots,s+k}\smallsetminus\set{i_1,i_2,\ldots,i_r}$ if and only if $\mu(v)=p_{i_1}p_{i_2}\cdots p_{i_r}$ and for each $j=1,2,\ldots, s$, $v\in S_j$ if and only if $\mu(v)=1$.
\end{lem}

\begin{proof}
We first note that $S_j$ is a clique containing only the vertex $v_j$ for $j=1,2,\ldots,s$. Now for each $j\in\set{1, 2,\ldots,k}$, $p_j$ divides $\lambda_i$ for some $i\in\set{s+1,s+2,\ldots,m}$ as $p_j$ divides $\lambda(\sigma)$. Thus $v_i\in S_{s+j}$. So $S_j\neq\emptyset$, for all $j=1, 2, \ldots s+k$. Also for any two vertices $v_r,v_t\in S_{s+j}$, $p_j$ divides both $\lambda_r$ and $\lambda_t$. So $v_r$ and $v_t$ are adjacent in $G$. Thus each $S_j$ is a clique of $G$. Again since $\{p_1, p_2, \ldots, p_k\}$ is the set of all prime factors of $\lambda(\sigma)$, each vertex $v\in V$ belongs to $S_j$ for some $j\in\set{1,2,\ldots,s+k}$. So $S$ covers $V$. Moreover, two vertices $v_i$ and $v_j$ are adjacent if and only if $\gcd(\lambda_i, \lambda_j) >1$. So $\lambda_i$ and $\lambda_j$ must have at least one common prime factor, say, $p_r$ and hence $v_i$ and $v_j$ both lie in $S_{s+r}$. So $S$ also covers $E$ and hence $S$ is a total clique covering of $G$. Now $v\in S_{i_1}\cap S_{i_2}\cap\cdots\cap S_{i_r}$ ($i_t>s$, $t=1,2,\ldots,r$) and $v\notin S_j$ for all $j\in\set{s+1,s+2,\ldots,s+k}\smallsetminus\set{i_1,i_2,\ldots,i_r}$ if and only if $p_{i_1}p_{i_2}\cdots p_{i_r}$ divides $\mu(v)$ but $p_j$ does not divide $\mu(v)$ for all $j\in\set{s+1,s+2,\ldots,s+k}\smallsetminus\set{i_1,i_2,\ldots,i_r}$. Now since non-trivial entries of $\sigma$ are square-free integers, $\mu(v)$ is a product of distinct primes and hence $\mu(v)=p_{i_1}p_{i_2}\cdots p_{i_r}$. By Definition \ref{dcseq}, for each $j=1,2,\ldots, s$, $v\in S_j$ if and only if $\mu(v)=1$.
\end{proof}

\begin{defn}\label{tcsigma}
Let $\sigma=(\lambda_1,\lambda_2, \ldots, \lambda_m)$ be a coding sequence of a graph $G$. Then the total clique covering $S$ of $G$ as defined in Lemma \ref{lem1} is called the {\em total clique covering of} $G$ {\em corresponding to} the sequence $\sigma$ and is denoted by $S_\sigma$.
\end{defn}

\begin{lem}\label{l1}
Let $G= (V, E)$ be a graph and $S$ be a total clique covering of $G$. Then there exists a coding sequence $\sigma$ of $G$ such that $S=S_\sigma$.
\end{lem}

\begin{proof}
Let $S=\{C_1, C_2, \ldots, C_{s+k}\}$, where $|C_i|=1$ for $i=1,2,\ldots,s$, $s\geqslant 0$. Let $p_i$ denote the $i^{\text {th}}$ prime. We define a map $f : V \rightarrow \mathbb N$ by $f(v)=1$ if $v\in C_1\cup C_2\cup\cdots\cup C_s$ and $f(v)=p_{i_1}p_{i_2}\cdots p_{i_r}$ if  $v\in C_{s+i_1}\cap C_{s+i_2}\cap\cdots\cap C_{s+i_r}$(where $1 \leq r \leq k$)  and $v\notin C_j$ for all $j\in\set{s+1,s+2,\ldots,s+k}\smallsetminus\set{s+i_1,s+i_2,\ldots,s+i_r}$. Now $S$ being a total clique covering, every vertex is assigned a label in this way. Also note that two vertices $u, v$ are adjacent in $G$ if and only if they lie in some $C_{s+j}\in S$ and consequently, if and only if $f(u)$ and $f(v)$ have a common prime factor $p_j$. Let us arrange the vertices of $G$ according to the non-decreasing order of $\Set{f(v)}{v\in V}$ and define the sequence $\sigma$ by $\sigma= (f(v_1), f(v_2), \ldots, f(v_m))$, where $f(v_1) \leq f(v_2) \leq \cdots \leq f(v_m)$. Then clearly $\sigma$ is a coding sequence of $G$ and $S=S_\sigma$.
\end{proof}

\begin{thm}\label{ktheta}
Let $G$ be a graph with $s\geqslant 0$ isolated vertices and $k$ be the minimum number of prime factors of $\lambda(\sigma)$ among all coding sequences $\sigma$ of $G$. Then $\theta_t(G)=k+s$.
\end{thm}

\begin{proof}
Let $\sigma$ be a coding sequence of $G$ such that the number of prime factors of $\lambda(\sigma)$ is $k$. Then by Lemma \ref{lem1}, $k+s=|S_\sigma|\geqslant\theta_t(G)$. Again let $S$ be a total clique covering of $G$ such $|S|=\theta_t(G)$. Then by Lemma \ref{l1}, there is a coding sequence $\sigma$ of $G$ such that $S=S_\sigma$ and it follows from the proof of Lemma \ref{l1} that there are $|S|-s$ prime factors of $\lambda(\sigma)$. Thus $k\leqslant |S|-s=\theta_t(G)-s$. Therefore $\theta_t(G)=k+s$.
\end{proof}

\begin{defn}\label{dcode}
Let $G=(V,E)$ be a graph with $|V|=m$ and $s\geqslant 0$ isolated vertices. Let $k=\theta_t(G)-s$. Let $S=\set{C_1,C_2,\ldots,C_{k+s}}$ be a total clique covering of $G$ such $|S|=\theta_t(G)$ and $|C_i|=1$ for $i=1,2,\ldots,s$. Now as in the proof of Lemma \ref{l1}, there are $k!$ ways of assigning first $k$ primes to the cliques $C_{s+1},C_{s+2},\ldots,C_{s+k}$ to obtain at most $k!$ different coding sequences. Let $\sigma[S]$ be the least among them in the lexicographic ordering in $\Real^m$. Then $\sigma[S]$ is called the {\em coding sequence of $G$ with respect to $S$}. Let $\{S_1, S_2, \ldots, S_r\}$ be the set of all total clique coverings of $G$ such that $|S_i|= \theta_t(G)$ for all $i=1, 2, \ldots, r$. Let $\sigma(G)$ be the least element of $\Set{\sigma[S_i]}{i=1, 2, \ldots,r}$ in the lexicographic ordering in $\Real^m$. Then $\sigma(G)$ is called the {\em code} of the graph $G$.
\end{defn}

\noindent
For example, $(2,3,3,5,5,6,6,6,7,11,385)$ and $(2,3,3,5,5,7,10,10,10,11,231)$ are two coding sequences of the graph $G$ in Example \ref{e1} with respect to the total clique covering $S=\{C_1, C_2, \ldots, C_5\}$. It is easy to see that $\theta_t(G)=5$ and $S$ is the only total clique covering with 5 cliques. The code of $G$ is $(2,2,3,3,5,7,10,10,10,11,231)$. Note that for any  graph $G$, we have $G \cong G[\sigma(G)]$.

\begin{thm}\label{unique}
Let $G_1$ and $G_2$ be two graphs. Then $G_1 \cong G_2$ if and only if $\sigma(G_1)=\sigma(G_2)$.
\end{thm}

\begin{proof}
The proof follows from Definition \ref{dcode}.
\end{proof}

\begin{defn}\label{dpol}
Let $G=(V,E)$ be a graph with $s\geqslant 0$ isolated vertices. Let $S=\set{C_1,C_2,\ldots,C_{s+k}}$ be a total clique covering of $G$, where $|C_i|=1$ for $i=1,2,\ldots,s$. If $v\notin C_1\cup C_2\cup\cdots\cup C_s$ , let $m(v)=\prod\Set{x_j}{v\in C_{s+j}}$ be a monomial in the polynomial semiring $\mathbb{Z}_0^+[x_1,x_2,\ldots,x_k]$ of $k$ variables over the semiring of non-negative integers with usual addition and multiplication. For any $v \in C_i$, $i=1, 2, \ldots, s$, define $m(v)=1$. Now we define
$$f(G, S)=f(x_1,x_2,\ldots,x_k)=\sum\limits_{v\in V} m(v).$$
Then $f(G, S)$ is said to be a {\em polynomial representation} of $G$ with respect to $S$.
\end{defn}

\noindent
Consider the graph $G$ in Example \ref{e1}. Then $f(G, S)=x_1+2x_2+2x_3+x_4+x_5+3x_1x_2+x_3x_4x_5$ is a polynomial representation of $G$ with respect to $S=\set{C_1,C_2,C_3,C_4,C_5}$. Now let $G=(V,E)$ be a graph and $\sigma=(\lambda_1,\lambda_2,\ldots,\lambda_m)$ be a coding sequence of $G$. From the construction of $S_{\sigma}$ (cf. Lemma \ref{lem1}, Definition \ref{tcsigma}) and by Definition \ref{dpol} it follows that $f(G, S_{\sigma})$ can also be obtained from $\sigma$ by replacing primes $p_i$ by $x_i$ ($i<j \Leftrightarrow p_i<p_j$) and commas by the addition symbol. It is important to note that the constant term of $f(G,S)$ in Definition \ref{dpol} is the number of isolated vertices of $G$.

\begin{defn}\label{dnpoly}
Let $G$ be a graph and $\sigma(G)$ be the code of $G$. Then the polynomial representation of $G$ corresponding to $\sigma(G)$, i.e., $f(G, S_{\sigma(G)})$ is called the {\em normal polynomial representation} or the {\em canonical polynomial representation} of $G$ and is denoted by $F(G)$.
\end{defn}

\noindent
The normal polynomial representation of the graph $G$ in Example \ref{e1} is given by 
$$F(G)=2x_1+2x_2+x_3+x_4+x_5+3x_1x_3+x_2x_4x_5.$$
The following interesting observations are immediate from Definition \ref{dnpoly}.

\begin{obs}\label{disconnected poly}
A graph $G$ is disconnected if and only if $F(G)= f_1 + f_2$, where $f_1$ and $f_2$ are polynomials with no common variables between them. The same is true for $f(G, S)$ for any total clique covering $S$ of $G$.
\end{obs}

\begin{obs}\label{bipartite poly}
A graph $G$ is bipartitie if and only if $F(G)=f_1+f_2$, where monomials belonging to the same $f_i$ have no common variables, for i=1,2. The same is true for $f(G, S)$ for any total clique covering $S$ of $G$.
\end{obs}

\noindent
We now proceed to obtain a formula for $F(G(n))$.

\begin{thm}\label{thetagn}
Let $n=p_1^{r_1}p_2^{r_2}\cdots p_k^{r_k}\in M$, where $p_i$'s are distinct primes. Then $\theta_t(G(n))= k$ and there is only one total clique covering of $G(n)$ with precisely $k$ cliques.
\end{thm}

\begin{proof}
If $n$ is prime then the result is obvious. Suppose $n$ is not a prime number. Now the vertices of $G(n)$ correspond to all the divisors of $n$, greater than $1$. So there are $k$ vertices labeled $p_i$, $i=1, 2, \ldots, k$. Clearly, no two of them can lie in the same clique as $\gcd(p_i, p_j)=1$  for all $i \neq j$. So in any total clique covering of $G(n)$, these $k$ vertices will be in $k$ different cliques. In other words, any total clique covering has at least $k$ cliques. Now let $C_i=\Set{v \in V(G(n))}{p_i \text{ divides } v}$. Then $S=\Set{C_i}{i=1, 2, \ldots,k}$ is easily seen to be a total clique covering of $G(n)$. So we have a total clique covering consisting of precisely $k$ cliques. Hence $\theta_t(G(n))=k$.

\noindent
Now we show that there is only one total clique covering of $G(n)$ containing precisely $k$ cliques. Suppose there is another total clique covering $D= \set{D_1, D_2, \ldots, D_k}$ of $G$ with exactly $k$ cliques. Here also, the vertices labeled $p_i$, $i=1, 2, \ldots, k$, are in distinct cliques. Without loss of generality, let $p_i \in D_i$ for $i=1, 2, \ldots, k$. Now for any $i$, $p_i$ is not adjacent to   those vertices which are not in $C_i$. So $D_i \subseteq C_i$ for all $i=1, 2, \ldots, k$. Now suppose that for some $j$, there is a vertex $v$ in $C_j$ which is not in $D_j$. Now $D$ being a total clique covering, $v$ has to lie in at least one $D_l$, where $l \not=j$. Since a total clique covering covers all the edges, the edge between $v$ and $p_j$ will be covered. $v \not\in D_j$ implies that $v$ and $p_j$ both lie in some $D_l$, where $l \neq j$. But this contradicts the fact that $p_j$ is not adjacent to $p_l$. So $D_i= C_i$ for all $i=1, 2, \ldots, k$ and hence $D=S$.
\end{proof}

\noindent
Let $\alpha(G)$ denote the maximum size of an independent set in a graph $G$. In general $\theta_t(G) \geqslant \alpha(G)$ for any  graph $G$. The proof of the following proposition is similar to that of Theorem \ref{thetagn} and so omitted.

\begin{prop}\label{prop1}
Let $G$ be a graph with $\alpha(G)=k$. If $\{v_1, v_2,\ldots, v_k\}$ is an independent set and $S=\{S_1, S_2, \ldots, S_k\}$ is a total clique covering of $G$ such that each $v_i$ lies only in $S_i$ among cliques in $S$ for each $i=1, 2, \ldots, k$, then $\theta_t(G)=\alpha(G)=k$ and $S$ is the only total clique covering containing exactly $k$ cliques.
\end{prop}

\noindent
Now we provide a formula for $F(G(n))$.

\begin{thm}\label{fgn}
Let $n=p_1^{r_1}p_2^{r_2}\cdots p_k^{r_k}\in M$, where $p_i$'s are distinct primes and $r_1\geqslant r_2\geqslant\cdots\geqslant r_k\geqslant 1$, $k\geqslant 1$. Then $F(G(n))$ contains all the monomials $x_{i_1}x_{i_2}\cdots x_{i_s}$, where $\set{i_1,i_2,\ldots,i_s}\subseteq\set{1,2,\ldots,k}$, $s\geqslant 1$ with the coefficient $r_{i_1}r_{i_2}\cdots r_{i_s}$, i.e.,
$$F(G(n))=\sum\limits_{\atop{i_1<i_2<\cdots <i_s}{\emptyset\neq\set{i_1,i_2,\ldots,i_s}\subseteq\set{1,2,\ldots,k}}} r_{i_1}r_{i_2}\cdots r_{i_s}\ x_{i_1}x_{i_2}\cdots x_{i_s}.$$
unless $n$ is prime. If $n$ is prime, then $F(G(n))=1$.
\end{thm}

\begin{proof}
If $n$ is prime, then $G(n)\cong K_1$ and so $F(G(n))=1$. Suppose $n$ is not prime. Now by Theorem \ref{thetagn}, $\theta_t(G(n))=k$ and there is only one total clique covering, say, $S$ of $G(n)$ with precisely $k$ cliques. Let $S=\{C_1, C_2, \ldots, C_k\}$ where $C_i=\Set{v \in V(G(n))}{p_i\text{ divides } v}$ for $i=1, 2, \ldots, k$. So $\sigma(G(n))$ is $\sigma[S]$.  First, let $k>1$.  Now $\sigma[S]$ is the least (in the lexicographic ordering) among the coding sequences of $G$ obtained from $S$. By Definition \ref{dcode}, any such coding sequence involves the primes $\{t_1, t_2, \ldots, t_k\}$, where $t_i$ is the $i^{\text {th}}$ prime. Now for any $p_i$, there are $r_i$ elements (namely, $p_i, p_i^2, \ldots, p_i^{r_i}$) which belongs to only the clique $C_i$. So considering the labellings and the fact that $r_1\geqslant r_2\geqslant\cdots\geqslant r_k\geqslant 1$, it is easy to see that the coding sequence will be the least in the lexicographic ordering if we assign $t_i$ to the clique $C_i$. So the vertex corresponding to a number $p_{i_1}^{q_{i_1}}p_{i_2}^{q_{i_2}}\cdots p_{i_s}^{q_{i_s}}$, ($1 \leqslant q_{i_j} \leqslant r_{i_j}$ for $j=1, 2, \ldots, s$ where $1 \leqslant s \leqslant k$), is assigned the label $t_{i_1}t_{i_2} \cdots t_{i_s}$. In other words, to find out $\sigma(G(n))$, we first consider the set of all divisors $\{m_1, m_2, \ldots, m_r\}$ (say) of $n$, greater than $1$, where $r=(r_1+1)(r_2+1)\cdots (r_k+1)-1$. Then in each $m_j$, we first replace $p_i$ by $t_i$, then make the resultant entries square-free and arrange them in non-decreasing order. This gives us $\sigma(G(n))=(\lambda_1,\lambda_2,\ldots,\lambda_r)$, where each $\lambda_i$ is a product of primes of the form $t_{i_1}t_{i_2}\cdots t_{i_s}$, ($1\leqslant s\leqslant k$) and this particular number repeats, say, $\nu$ times in the sequence $\sigma(G(n))$, where $\nu$ is the number of  divisors of $n$ (greater than $1$) which are of the form $p_{i_1}^{q_{i_1}}p_{i_2}^{q_{i_2}}\cdots p_{i_s}^{q_{i_s}}$, $1 \leqslant q_{i_j} \leqslant r_{i_j}$ for $j=1, 2, \ldots, s$, i.e., $\displaystyle{\nu=\prod\limits_{j=1}^s r_{i_j}}$. 

\noindent
Thus by Definition \ref{dnpoly} we have $F(G(n))=\sum\limits_{\atop{i_1<i_2<\cdots <i_s}{\emptyset\neq\set{i_1,i_2,\ldots,i_s}\subseteq\set{1,2,\ldots,k}}} r_{i_1}r_{i_2}\cdots r_{i_s}\ x_{i_1}x_{i_2}\cdots x_{i_s}.$  

\vspace{0.5em}\noindent
Finally, if $k=1, r_1>1$, then $G(n)\cong K_{r_1}$. So $F(G(n))=r_1x_1$, which satisfies the above formula. This completes the proof.
\end{proof}

\noindent
For example, $F(G(p^2qr))=F(G(60))=2x_1+x_2+x_3+2x_1x_2+2x_1x_3+x_2x_3+2x_1x_2x_3$, where $p,q,r$ are distinct primes.  Further one may easily verify codes and normal polynomial representations for the following special classes of graphs. 

\begin{enumerate}
\item[(i)] $\sigma(K_1)=(1)$, $F(K_1)=1$, $\sigma(K_n)=(2,2,\ldots,2)$ ($n$ times), $F(K_n)=nx_1$ for $n\geqslant 2$,
\item[(ii)] $\sigma(P_n)=(p_1, p_2, p_1p_3, p_2p_4, \ldots, p_{n-3}p_{n-1}, p_{n-2}p_{n-1})$, 
$F(P_n)=x_1+ x_2 + \displaystyle \sum\limits_{i=1}^{n-3}x_ix_{i+2}+ x_{n-2}x_{n-1}$ for $n\geqslant 3$,
\item[(iii)] $\sigma(C_n)=(p_1p_2, p_1p_3, p_2p_4, \ldots, p_{n-2}p_n, p_{n-1}p_n)$, 
$F(C_n)=x_1x_2 +x_1x_3 +\displaystyle \sum\limits_{i=1}^{n-2}x_ix_{i+2}+x_{n-1}x_n$ for $n\geqslant 4$,
\end{enumerate}

\noindent
where $p_i$ is the $i^\text{th}$ prime for $i=1,2,\ldots,n$ and $K_n$, $P_n$ and $C_n$ are respectively the complete graph, the path and the cycle with $n$ vertices.

\vspace{1em}
\noindent
{\Large \bf Conclusion}

\vspace{1em}
\noindent
There are various representations of simple undirected graphs in terms of adjacency matrices, adjacency lists, unordered pairs etc. But none of them is unique for isomorphic graphs except the one described in the \texttt{nauty} algorithm by McKay \cite{MK}. But this is, in fact, a {\em canonical isomorph} \cite{HR} which is a graph rather than a sequence of integers as the code of a graph introduced here. The importance of the code $\sigma(G)$ is its uniqueness and its simple form. It is same for any set of isomorphic graphs. The determination of $\sigma(G)$ is not always easy, but once it is obtained for a graph it becomes the characteristic of the graph. Authors believe that further study of the code and the normal polynomial representation of a simple undirected graph will be helpful in further research on graph theory. The purpose of this note is to communicate these interesting observations to all graph theorists.


\small

\begin{thebibliography}{99}

\bibitem{CGMS} Ivy Chakrabarty, Shamik Ghosh, T.~K. Mukherjee and M.~K. Sen, Intersection graphs of ideals of rings, Discrete Math. \textbf{309} (2009), 5381--5392.

\bibitem{MK} B.~D. McKay, Practical graph isomorphism, Congr. Numer., \textbf{30} (1981), 45--65.
\bibitem{HR} Stephen G. Hartke and A.~J. Radcliffe, McKay's canonical graph labeling algorithm, Contemporary Mathematics, \textbf{479} (2009), 99--111.

\end{thebibliography}
\end{document}